\documentclass[11pt,reqno]{amsart}
\usepackage[utf8]{inputenc}
\usepackage{amsmath, amssymb, amsthm, mathrsfs,mathtools}
\usepackage{enumerate} 
\usepackage{comment,relsize,braket}
\usepackage{color}  
\usepackage{graphicx}
\usepackage{tikz}
\usepackage[top=1in, bottom=1.3in, left=1in, right=1in]{geometry}

\usepackage[all,2cell,ps]{xy}
\usepackage[pagebackref,colorlinks, linkcolor=blue, citecolor=red]{hyperref}

\usepackage{fancyhdr}
\newcommand{\ncom}{\newcommand}
\ncom{\ov}{\overline}
\ncom{\N}{\mathbb N}
\ncom{\Q}{\mathbb Q}
\ncom{\Z}{\mathbb Z}
\ncom{\F}{\mathbb{F}} 
\ncom{\C}{\mathbb{C}} 
\ncom{\p}{\mathfrak{p}} 
\ncom{\q}{\mathfrak{q}} 
\ncom{\m}{\mathfrak{m}}
\ncom{\n}{\mathfrak{n}}
\ncom{\kk}{\mathbf{k}}
\ncom{\mL}{\mathcal{L}} 
\ncom{\mN}{\mathcal{N}} 
\ncom{\f}{\mathcal{F}}
\ncom{\g}{\mathcal{T}}
\ncom{\I}{{\mathcal{I}}}
\ncom{\J}{\mathcal{J}} 
\ncom{\R}{\mathcal{R}}
\ncom{\h}{{\bf h}}

\DeclareMathOperator{\depth}{depth}
\DeclareMathOperator{\chr}{char}

\DeclareMathOperator{\ini}{in}

\DeclareMathOperator{\gr}{gr}
\DeclareMathOperator{\HH}{H}

\def\thebibliography#1{\section*{References}
	\list{[\arabic{enumi}]}{\settowidth \labelwidth{[#1]} \leftmargin
		\labelwidth \advance \leftmargin \labelsep \usecounter{enumi}}
	\def\newblock{\hskip .11em plus .33em minus .07em} \sloppy
	\clubpenalty 4000 \widowpenalty 4000 \sfcode`\.=1000 \relax}

\newtheorem{Theorem}{Theorem}[section]
\newtheorem{Lemma}[Theorem]{Lemma}
\newtheorem{Corollary}[Theorem]{Corollary}
\newtheorem{Proposition}[Theorem]{Proposition}
\newtheorem{Remark}[Theorem]{Remark}

\newtheorem{Example}[Theorem]{Example}

\newtheorem{Definition}[Theorem]{Definition}

\numberwithin{equation}{section}

\begin{document}
	
\pagenumbering{arabic}  
\title[]{Tight Closure of powers of ideals and \\ [2mm] Tight Hilbert Polynomials}
\thanks{The first author is supported by a UGC fellowship, Govt. of India}

\author[]{Kriti Goel}
\author[]{Vivek Mukundan}
\author[]{J. K. Verma}
\address{Indian Institute of Technology Bombay, Mumbai, INDIA 400076}
\address{University of Virginia, Charlottesville, VA 22904, USA}
\email{kriti@math.iitb.ac.in}
\email{vm6y@eservices.virginia.edu}
\email{jkv@math.iitb.ac.in}

\maketitle

\thispagestyle{empty}

\begin{abstract}
Let $(R,\m)$ be an analytically unramified local ring of positive prime characteristic $p.$ For an ideal $I$, let $I^*$ denote its tight closure. We introduce  the {\em tight Hilbert function }$H^*_I(n)=\ell(R/(I^n)^*)$ and the corresponding {\em tight Hilbert polynomial } $P_I^*(n)$, where $I$ is an $\m$-primary ideal. It is proved that $F$-rationality can be detected by the vanishing of the first coefficient of $P_I^*(n).$ We find the tight Hilbert polynomial of certain parameter ideals in hypersurface rings and Stanley-Reisner rings of simplicial complexes.
\end{abstract}

\section{Introduction}

Let $(R,\m)$ be a $d$-dimensional Noetherian local ring and $I$ be an $\m$-primary ideal. Let $\ov{I}$ be the integral closure of $I.$ The Rees algebra of $I$ is denoted by $\mathcal R(I)=\oplus_{n\in \Z}I^nt^n$ where $t$ is an indeterminate. The integral closure  of $\mathcal R(I)$ in $R[t,t^{-1}]$ is $\ov{\mathcal R}(I)=\oplus_{n\in \Z} \overline{I^n}t^n.$ We use  $\ell(M)$ to denote the  length  of an $R$-module $M.$ David Rees \cite{reesAUrings} showed that if $R$ is analytically unramified then $\ov{\mathcal R}(I)$ is a finite module over $\mathcal R(I).$ This implies that for all large $n,$ the {\em normal Hilbert function of $I,$} $\ov{H}_I(n)=\ell(R/\ov{I^n})$ is a polynomial of degree $d.$  This is called the {\em normal Hilbert polynomial of $I$} and it  is denoted by $\ov{P}_I(n).$  We write
\[ \ov{P}_I(n)=\ov{e}_0(I)\binom{n+d-1}{d}-\ov{e}_1(I)\binom{n+d-2}{d-1} +\dots+(-1)^d\ov{e}_d(I).\]
Here $\ov{e}_0(I)=e(I),$  the multiplicity of $I$ and the coefficients $\ov{e}_i(I)$ for $i=0,1,\dots, d$ are called the {\em normal Hilbert coefficients of $I.$}
 The  normal Hilbert coefficients  play an important role in the study of singularities of algebraic varieties. Rees \cite{rees1981hilbert} proved that if $(R,\m)$ is a $2$-dimensional analytically unramified normal  local ring then it is pseudo-rational if and only if $\ov{e}_2(I)=0$ for all $\m$-primary ideals $I$ of $R.$ 

Shiroh Itoh \cite{itoh92} proved that if $R$ is Cohen-Macaulay and analytically unramified, then $\ov{e}_3(I)\geq 0.$ Moreover, if $R$ is Gorenstein and $I$ is generated by a regular sequence so that $\ov{I}=\m$, then $\ov{e}_3(I)=0$ if and only if $\ov{r}(I)\leq 2.$ Here $\ov{r}(I)=\min\{n\mid I \ov{I^n}=\ov{I^{n+1}}\}.$ The integer $\ov{r}(I)$ is called the {\em normal reduction number of $I.$ }
Itoh \cite{itoh92} proved that if $\ov{r}(I)\leq 2$, then $\ov{\mathcal R}(I)$ is Cohen-Macaulay. He conjectured that if $\ov{e}_3(I)=0$, then $\ov{r}(I)\leq 2.$  This is known in some cases \cite{cpr}, \cite{kuma}, however, it remains open in the general case.

Now let $R$ have prime characteristic $p> 0.$ Let $I^*$ denote the tight closure of $I.$ Then $I\subset I^*\subset \ov{I}.$ If $(R,\m)$ is analytically unramified, then the filtration $\{(I^n)^*\}$ is also $I$-admissible, i.e., the Rees algebra $\mathcal R^*(I)=\oplus_{n\in \Z}(I^n)^* t^n$ is a finite $\mathcal R(I)$-module. Therefore, the 
{\em tight Hilbert function of $I,$} namely, $H^*_I(n)=\ell(R/(I^n)^*)$ is a polynomial of degree $d$ for large $n.$ We call  this polynomial  as the {\em tight Hilbert polynomial of $I$ } and denote it by $P^*_I(n).$ We write the tight Hilbert polynomial as

\[ P^*_I(n)=e(I)\binom{n+d-1}{d}-e^*_1(I)\binom{n+d-2}{d-1}+\dots+(-1)^de^*_d(I).\]

One of the main objectives of this paper is to
initiate a study of the tight Hilbert polynomial since preliminary results obtained in this paper indicate its close connection  with $F$-singularities. Recall that $R$ is called an $F$-rational local ring if ideals of principal class in $R$ are tightly closed.  We show that
if $R$ is a Cohen-Macaulay analytically unramified local ring, then $R$ is $F$-rational if and only if $e_1^*(I)=0$ for some parameter ideal $I$ of $R.$ It is, therefore, reasonable to expect that the other coefficients of the tight Hilbert polynomial of $I$ may have a close connection with the properties of $R.$ We, therefore, calculate the tight Hilbert polynomial of parameter ideals in certain hypersurface rings and Stanley-Reisner rings of simplicial complexes.

Here is a summary of the contents of the paper. In section 2, we set the notation and recall certain definitions and results needed in later sections.  In section 3, we introduce the condition $HI_p$.

\begin{Definition}{\rm
	Let $R$ be a $d$-dimensional  Noetherian local ring. Let $I$ be an ideal generated by an $R$-regular sequence. An $I$-admissible filtration $\mathcal{F}=\{I_n\}$ is said to satisfy \emph{the condition $HI_{p}$} if for all $n \geq p$, 
	\begin{align*}
	I_{n+1}\cap I^{n-p}=I_{p+1}I^{n-p}.
	\end{align*}
}\end{Definition}
Let $\R(\f)$ denote the Rees algebra $\oplus_{n\in \Z}I_nt^n$ of the filtration $\f=\{I_n\}$. Suppose $I=(x_1,\dots,x_d)$, where $x_1,\dots,x_d$ is an $R$-regular sequence and let $J=(t^{-1},x_1t,\dots,x_dt)$ be an $\R(I)$-ideal. Itoh proved that for the  $I$-admissible filtration $\f=\{\ov{I^n}\}$, vanishing of $\HH^2_J(\R(\f))_0$ implies the intersection theorem $\ov{I^{n+1}}\cap I^n = I^n\ov{I}$, for all $n\ge 1$. Let $r(\f)$ denote the reduction number of the filtration $\f$. We find a formula of the Hilbert-Samuel polynomial of an $I$-admissible filtration $\f$ which satisfies the condition $HI_p$, for all $p\le r(\f)-2$.
\begin{Theorem}
	Let $R$ be a $d$-dimensional Cohen-Macaulay Noetherian local ring and $I$ be an ideal generated by an $R$-regular sequence. Let $\f=\{I_n\}$ be an $I$-admissible filtration. If $\f$ satisfies the condition $HI_p$, for $p\leq r=r(\mathcal{F})-2$, then for all $i=1,\dots, d$,
	\begin{align*}
	e_i(\f)=\sum_{k=i-1}^{\infty}\binom{k}{i-1}\ell\left(\frac{I_{k+1}}{II_k}\right).
	\end{align*}
\end{Theorem}

Let $(R,\m)$ be a $d$-dimensional analytically unramified local ring and let $I$ be an $\m$-primary ideal. In section 4, we prove that for large $n$, $H_I^*(n) = \ell(R/(I^n)^*)$ is a polynomial, $P_I^*(n)$, of degree $d$ and with coefficients in $\Q$.  We shall prove:
\begin{Theorem}
	Let $(R,\m)$ be a $d$-dimensional analytically unramified Cohen-Macaulay local ring with positive prime characteristic $p$. Then $R$ is $F$-rational if and only if $e_1^*(I)=0,$ for some parameter ideal $I$.
\end{Theorem}
We shall prove that in positive prime characteristic $p$, the Huneke-Itoh intersection theorem is indeed true for tight closure filtration of a complete intersection. 
\begin{Proposition}
	Let $R$ be a Noetherian ring and I be an ideal generated by a regular sequence. Then for all $n \ge 1$, 
	\[ I^n \cap {(I^{n+1})^*}=I^n {I^*}. \]
\end{Proposition}
The above proposition  for tight closures of powers of $I$ will enable us to find the tight Hilbert polynomial of $I$ when the tight reduction number is at most two. 

In section 5,  we calculate the tight Hilbert polynomial in diagonal hypersurface rings. Let $R=\F[[X,Y,Z]]/(X^N+Y^N+Z^N)$, where $N \ge 2$ and $\F$ is an infinite field of characteristic $p$ and $p \nmid N$. Let $x,y,z$ denote the images of $X,Y,Z$ respectively in $R$. Let $I=(y,z)$ and  $\m=(x,y,z)$. We show that if $N=2$, then $R$ is $F$-rational and if $N \ge 3$, then for all $k \ge 1$,
\[ (I^k)^*=\m^{k+1}+I^k \]
and the tight reduction number of $I$, $r^*(I)=N-2$.	

In section 6, we calculate tight closure of power product of ideals generated by linear systems of parameters in Stanley-Reisner rings of simplicial complexes. Let $\Delta$ be a $(d-1)$-dimensional simplicial complex on $n$ vertices and $\kk$ be a field of char $p>0$. Then $R= \kk[\Delta]=\kk[X_1,\dots ,X_n]/I_\Delta$ is a $d$-dimensional ring. Let $x_1,\dots,x_n$ denote the images of $X_1,\dots,X_n$ respectively in $R$, $\m=(x_1,\dots ,x_n)$ be the unique maximal homogeneous ideal of $\kk[\Delta]$ and $I_1,\dots,I_g$ be ideals generated by linear systems of parameters. Let $s_1,s_2,\dots, s_g\in \mathbb{N}$. We prove that 
\[ \ov{I_1^{s_1}I_2^{s_2}\cdots I_g^{s_g}} = (I_1^{s_1}I_2^{s_2}\cdots I_g^{s_g})^* = \m^{s_1+s_2+\cdots+s_g}. \]

Let $h(\Delta)=(h_0, h_1,\dots ,h_d)$ denote the $h$-vector of $\kk[\Delta]$. The Hilbert series of $R$ is of the form
\[ H(\kk[\Delta],\lambda) = \frac{h_0+h_1\lambda+\cdots +h_d\lambda^d}{(1-\lambda)^d}. \]
Set $\h(\lambda)=h_0+h_1\lambda+\cdots +h_d\lambda^d$. Let $\h^{(i)}(\lambda)$ denote the $i$-th derivative of $\h(\lambda)$ with respect to $\lambda$. Let $s_1,s_2,\dots, s_g\in \mathbb{N}$ with $s_1+s_2+\cdots+s_g = n+1$. Then 
\[ \ell\left(\frac{R}{(I_1^{s_1}I_2^{s_2}\cdots I_g^{s_g})^*}\right) = \ell\left(\frac{R}{\m^{n+1}}\right) = \sum_{i=0}^{d}(-1)^i\frac{\h^{(i)}(1)}{i!} {n+d-i\choose d-i}. \]
In particular, $e_d^*(I_1^{s_1}I_2^{s_2}\cdots I_g^{s_g}) = h_d$, for any $s_1,s_2,\dots,s_g\in \mathbb{N}$. 

We generalize a theorem of Rees  \cite{rees1981hilbert} for Stanley-Reisner rings. Recall that for a simplicial complex $\Delta$ with $f$-vector $f(\Delta) = (f_{-1},f_0,\dots,f_{d-1})$, the \emph{Euler characteristic} of $\Delta$ is denoted by 
\[ \chi(\Delta)=f_0-f_1+\cdots+(-1)^{d-1}f_{d-1}. \]
The simplicial complex $\Delta$ is said to be \emph{Eulerian} if $\chi(\Delta)=1$. Set $S=(s_1,s_2,\dots ,s_d)$, $S_i=(s_1,\dots ,s_{i-1},s_i-1,s_{i+1},\dots ,s_d)$ and $\underline{I}^{S_i}=(\prod_{j\not= i}I_j^{s_j})I_i^{s_i-1}$, for all $i=1,2,\dots ,d$.

\begin{Theorem}
	Let $\kk[\Delta]$ be a $d$-dimensional Cohen-Macaulay Stanley-Reisner ring of a  simplicial complex $\Delta$. Let $\m$ be its unique maximal homogeneous ideal and $I_1,I_2,\dots ,I_d$ be ideals generated by linear system of parameters. Then for any joint reduction $(a_1,a_2,\dots ,a_d)$ of the $d$-tuple $(I_1,I_2,\dots,I_d)$, the following are equivalent:\medskip \\
	{\rm(1)} $\Delta$ is Eulerian.\medskip\\
	{\rm(2)} $r(\m)\le d-1.$\medskip\\
	{\rm(3)} $a_1 (\underline{I}^{S_1})^* + a_2(\underline{I}^{S_2})^* + \cdots + a_d(\underline{I}^{S_d})^* = (\underline{I}^S)^*$, for all $s_1,s_2,\dots ,s_d\ge 1$.\medskip\\
	{\rm(4)} $\displaystyle e^*_d(I_1\cdots I_d) -\sum_{i=1}^{d}\left(e^*_d \big(\prod_{j\not=i}I_j \big)\right)+\sum_{1\le p<q\le d}\left(e^*_d \big(\prod_{j\not=p,q}I_j \big)\right)-\cdots+(-1)^{d-1} \sum_{i=1}^{d}\left(e^*_d(I_i)\right)=0.$
\end{Theorem}

\noindent
{\bf Acknowledgements:} We thank K.-i. Watanabe  and Anurag Singh for numerous discussions.

\section{Preliminaries}
In this section, we recall a few definitions, results, and set up notation. 

\begin{Definition}{\rm
	Let $(R, \m)$ be a local ring and $I$ be an $\m$-primary ideal. Then a sequence of ideals $\f = \{I_n\}_{n\in\Z}$, where $I_n = R$ for all $n\le 0$, is called an \emph{$I$-filtration} if for all $n,m \in \Z$, we have
	{\rm(1)} $I_{n+1} \subseteq I_n$
	{\rm(2)} $I_m \cdot I_n \subseteq I_{m+n}$ and 
	{\rm(3)} $I^n \subseteq I_n.$\\
	An $I$-filtration $\f$ is called \emph{admissible} if in addition
	to the above properties, there exists a $k\in\N$ such that $I_n \subseteq I^{n-k}$, for all $n$.
}\end{Definition}

Let $t$ be an indeterminate over $R$. The \emph{Rees ring of $I$} is the graded ring $\oplus_{n\in \Z} I^n t^n$, denoted by $\R(I)$. For a filtration $\f=\{ I_n \}$, the \emph{Rees ring of $\f$} is $\R(\f)=\oplus_{n\in \Z}I_nt^n$. If $\f$ is an $I$-admissible filtration, then by definition, $\R(\f)$ is a finite $\R(I)$-module. Another blow-up algebra which we often refer to in this article is the associated graded ring of $I$, $\gr_I(R)=\bigoplus_{n\ge 0}I^n/I^{n+1}$ and the associated graded ring of $\f$, $\gr_{\f}(R) = \bigoplus_{n\ge 0}I_n/I_{n+1}$.

\begin{Proposition} \label{poly}
	Let $(R,\m)$ be a $d$-dimensional local ring and $I$ be an $\m$-primary ideal. Let  $\f = \{I_n\}$ be an $I$-admissible filtration. Then\\
	{\rm(1)} $\gr_{\f}(R)$ is a finite $\gr_I(R)$-module.\\
	{\rm(2)} $\dim \gr_{\f}(R) = d$.\\
	{\rm(3)} $H_{\f}(n) = \ell(R/I_n)$ is a polynomial function given by $P_{\f}(x) \in \Q[x]$. Moreover, $P_{\f}(x)$ and $P_I(x)$ have the same leading coefficient.
\end{Proposition}

Let $I$ be an $\m$-primary ideal. The \emph{Hilbert function} of the filtration $\f$ is defined by $H_{\f}(n) = \ell(R/I_n)$. If $\f$ is an $I$-admissible filtration, then for sufficiently large $n$, $H_{\f}(n)$ coincides with a polynomial 
\[ P_{\f}(n) = e_0(\f){n+d-1 \choose d} - e_1(\f){n+d-2 \choose d-1}+\cdots + (-1)^d e_d(\f) \]
of degree $d$, called the \emph{Hilbert polynomial} of $\f$. The coefficients $e_i(\f)$ for $i=1,\dots,d$ are integers, called the \emph{Hilbert coefficients} of $\f$.

\begin{Definition}{\rm
A \emph{reduction} of an $I$-admissible filtration $\f = \{I_n\}_{n\in\Z}$ is an ideal $J \subseteq I_1$ such that $JI_n = I_{n+1}$, for all large $n$. A minimal reduction of $\f$ is a reduction of $\f$ minimal with respect to containment. Let $J$ be a reduction of $\f$. The \emph{reduction number} of $\f$ with respect to $J$, denoted by $r_J(\f)$, is the smallest integer $n$ such that $JI_n = I_{n+1}$. The reduction number of $\f$ equals
\[ r(\f) = \min\{ r_J(\f) \mid J \text{ is a minimal reduction of } I \}. \]
}\end{Definition}

\begin{Definition}{\rm
Let $I_1,I_2,\dots,I_d$ be ideals in a ring $R$ and $x_i \in I_i$ for all $i =1,2,\dots,d$. The $d$-tuple $(x_1,x_2,\dots,x_d)$ is said to be a \emph{joint reduction} of the $d$-tuple $(I_1,I_2,\dots,I_d)$ if the ideal $\sum_{i=1}^{d} x_i I_1\cdots I_{i-1}I_{i+1}\cdots I_d$ is a reduction of $I_1 \cdots I_d$. In particular, for large $s_1,\dots,s_d$, we have
\[ I_1^{s_1}\cdots I_d^{s_d} = \sum_{i=1}^{d} x_i\underline{I}^{S_i}. \]
}\end{Definition}

In this paper, we focus on two closure operations, integral closure and tight closure.

\begin{Definition}{\rm
	Let $R$ be a $d$-dimensional Noetherian ring and let $I$ be an ideal of $R$. An element $x \in R$ is said to be \emph{integral over $I$} if there exists an equation of the form 
	\[ x^n+a_1x^{n-1}+\cdots + a_n=0 \]
	where $a_i\in I^i$, for all $1 \le i \le n$. The set $\ov{I}$ of all elements in $R$ which are integral over $I$ is an ideal, called the integral closure of $I$.
}\end{Definition}

\begin{Definition}{\rm
	A Noetherian local ring $(R,\m)$ is called \emph{analytically unramified} if its $\m$-adic completion is reduced.
}\end{Definition}
Rees characterized analytically unramified local rings in terms of integral closure of powers of ideals \cite{reesAUrings}.
\begin{Theorem}\label{reesau}
	Let $(R,\m)$ be a Noetherian local ring. Then the following are equivalent:\\
	{\rm(1)} $R$ is analytically unramified.\\
	{\rm(2)} For all ideals $I \subseteq R$, there is an integer $k$ such that $\ov{I^{n+k}} \subseteq I^n$, for all $n \ge 0$.\\
	{\rm(3)} There exist an $\m$-primary ideal $J$ and an integer $k$ such that for all $n \ge 0$, \[ \ov{J^{n+k}} \subseteq J^n. \]
	{\rm(4)} There exist an $\m$-primary ideal $J$ and a function $f : \N \rightarrow \N$ with
	$f(n)\rightarrow\infty$ as $n \rightarrow \infty$ such that $\ov{J^n} \subseteq J^{f(n)}$, for all $n \ge 0$.
\end{Theorem}

If $(R,\m)$ is an analytically unramified local ring and $I$ is an $\m$-primary ideal, then the $I$-filtration $\f=\{\ov{I^n}\}$ is admissible and $R(\f)$ is a finite $R(I)$-module. 

\begin{Definition}
	For a ring $R$, the subset of $R$ consisting of all the elements which are not contained in any minimal prime ideal of $R$, is denoted by $R^\circ$. 
\end{Definition}

Let $\chr R = p >0$. Consider the Frobenius map $F:R\rightarrow R$ defined as $F(x)  = x^p$. Then $F$ is a ring homomorphism. The $e$-th iterate of $F$ is defined as $F^e(x) = x^{p^e}$. Note that R may be viewed as an R-module via the Frobenius homomorphism or any of its iterations. We write $F^e(R)$ to denote $R$ viewed as an $R$-module via the ring homomorphism $F^e$. Let $q=p^e$. If $I=(a_1,\dots,a_r)$ is an ideal of $R$ then $I^{[q]}=(a_1^q,\dots,a_r^q)$. For any positive integer $n$, $(I^n)^{[q]}=(I^{[q]})^n$. 
\begin{Definition}
	An element $x \in R$ is said to be in the tight closure $I^*$ of $I$ if there exists $c \in R^\circ$ such that $cx^q \in I^{[q]}$, for all sufficiently large $q$.
\end{Definition}

Huneke and Itoh independently proved the theorem below for  integral closure of powers of ideals generated by a regular sequence.  While Huneke proved it for Cohen-Macaulay local rings containing a field, Itoh proved it in general for any Noetherian ring.
\begin{Theorem}[Itoh, {\cite[Theorem 1]{itoh88}}]\label{itohthm}
	Let $R$ be a Noetherian ring and $I$ be an ideal generated by an $R$-regular sequence. Then for all $n \ge 1$,
	\[ \ov{I^{n+1}} \cap I^n = \ov{I} I^n. \]
\end{Theorem}
Itoh proved this theorem as a consequence of the following vanishing theorem.
\begin{Theorem}\label{itohvanish}
	{\rm (\cite[Theorem 2]{itoh88},\cite[Theorem 1.2]{hongUlrich})} Let $R$ be a $d$-dimensional Noetherian ring and $I=(x_1,\dots,x_d)$ be an ideal generated by an $R$-regular sequence. Let  $\ov{\R}(I)=\oplus_{n\in \Z} \ov{I^n}t^n.$ Let $J=(t^{-1}, x_1t,\dots,x_rt)$, where $2 \le r \le d$. Then $\HH^2_J(\ov{\R}(I))_n=0$, for all $n\le 0$.
\end{Theorem}

As a consequence of the Huneke-Itoh intersection theorem, one can find the Hilbert-Samuel polynomial of the filtration $\f=\{\ov{I^n}\}$ if the normal reduction number of the filtration is at most 2. 
\begin{Proposition} \label{itoh}
	{\rm (\cite[Theorem 10]{itoh88})} Let $(R,\m)$ be  $d$-dimensional Cohen-Macaulay analytically unramified local ring and $I$ be an ideal generated by a system of parameters. Then $\ov{I^{n+1}} =\ov{I^2}I^{n-1}$ for all $n \geq 1$ if and only if for all $n\geq 1,$
	\[ \ell\left(\frac{R}{ \ov{I^{n+1}}}\right) = \ov{e}_0(I) \binom{n+d}{d} - \ov{e}_1(I) \binom{n+d-1}{d-1} + \ov{e}_2(I) \binom{n+d-2}{d-2} \]
	where 
	\[ \ov{e}_0(I)=\ell\left(\frac{R}{I}\right), \, \ov{e}_1(I)=\ell\left(\frac{\ov{I}}{I}\right) + \ell\left(\frac{\ov{I^2}}{I \ov{I}}\right), \, \ov{e}_2(I)=\ell\left(\frac{\ov{I^2}}{I\ov{I}}\right). \]
\end{Proposition} 

For all undefined terms, we refer to \cite{matsumuraCRT}.

\section{The condition $HI_p$}
Let $R$ be a $d$-dimensional  Noetherian local ring. Let $x_1,\dots,x_d$ be an $R$-regular sequence and $I=(x_1,\dots,x_d)$. Let $\f=\{I_n\}_{n \in \Z}$ be an $I$-admissible filtration, where $I_n=R$, for all $n \le 0$.

\begin{Definition}{\rm
	Let $R$ be a $d$-dimensional  Noetherian local ring. Let $I$ be an ideal generated by an $R$-regular sequence. An $I$-admissible filtration $\mathcal{F}=\{I_n\}$ is said to satisfy \emph{the condition $HI_{p}$} if for all $n \geq p$, 
	\begin{align}\label{HIpcond}
	I_{n+1}\cap I^{n-p}=I_{p+1}I^{n-p}.
	\end{align}
}\end{Definition}

\begin{Remark}\label{sec3:rmkOnReductionNumberImpliesHIp}{\rm
		For any $I$-admissible filtration $\f=\{I_n\}$ with reduction number $r$, the condition $HI_p$ is true when $p\ge r-1$. Observe that if $m \ge r$, then $I_{m+j} = I^jI_m$ for all $j \ge 0.$ Thus if $p \ge r-1$, then $n+1 = (p+1)+(n-p)$ and hence for all $n \ge p$,
		\[ I_{n+1} \cap I^{n-p} = I_{p+1}I^{n-p} \cap I^{n-p} = I_{p+1}I^{n-p}. \]
}\end{Remark}

A close examination of the proof of Proposition \ref{itoh} shows that the result is also true for any $I$-admissible filtration $\{I_n\}$ satisfying the condition $HI_0$. We prove a lemma first.

\begin{Lemma}\label{iso}
	Let $(R,\m)$ be  $d$-dimensional Cohen-Macaulay local ring and $I=(x_1,\dots,x_d)$ be an ideal generated by a system of parameters. Let $\f=\{I_n\}$ be an $I$-admissible filtration satisfying the condition $HI_{k-1}$, for some $k\ge 1$. Then $\text{for all } n\ge k$,
	\begin{align}\label{sec1:lem1:eq1}
	\frac{I^{n-k}I_{k+1}}{I^{n-k+1}I_k}\cong \frac{I^{n-k}}{I^{n-k+1}}\bigotimes_R\frac{I_{k+1}}{II_k}.
	\end{align}
	In particular for all $n \ge k$,
	\begin{align*}
	\ell\left(\frac{I^{n-k}I_{k+1}}{I^{n-(k-1)}I_k}\right)=\binom{n+d-k-1}{d-1}\ell\left(\frac{I_{k+1}}{II_k}\right).
	\end{align*}
\end{Lemma}

\begin{proof}
	The isomorphism in \eqref{sec1:lem1:eq1} will be proved once we establish the following $R$-module isomorphisms
	\begin{align*}
	\frac{I^{n-k}I_{k+1}}{I^{n-k+1}I_k}\cong \frac{I^{n-k}I_{k+1}+I^{n-k+1}}{I^{n-k+1}}, && \frac{I^{n-k}I_{k+1}+I^{n-k+1}}{I^{n-k+1}}\cong \frac{I^{n-k}}{I^{n-k+1}}\bigotimes_R\frac{I_{k+1}}{II_k}.
	\end{align*}
	To prove the first isomorphism notice that,
	\begin{align*}
	\frac{I^{n-k}I_{k+1}+I^{n-k+1}}{I^{n-k+1}}\cong\frac{I^{n-k}I_{k+1}}{I^{n-k}I_{k+1}\cap I^{n-k+1}}.
	\end{align*}
	Thus it is enough to show that for all $n \ge k$, $I^{n-k}I_{k+1}\cap I^{n-k+1}=I^{n-k+1}I_k$. But 
	\begin{align*}
	I^{n-k+1}I_k\subseteq I^{n-k}I_{k+1}\cap I^{n-k+1}\subseteq I_{n+1}\cap I^{n-(k-1)}=I_kI^{n-k+1}
	\end{align*}
	where the last equality is due to the hypothesis $I_{n+1}\cap I^{n-(k-1)}=I_{k}I^{n-(k-1)}$. Thus it follows that $I^{n-k}I_{k+1}\cap I^{n-k+1}=I^{n-k+1}I_k$. 
	
	Now we prove the second isomorphism. Consider the natural multiplication map 
	\begin{align*}
	\frac{I^{n-k}}{I^{n-k+1}}\bigotimes_R\frac{I_{k+1}}{II_k}\longrightarrow\frac{I^{n-k}I_{k+1}+I^{n-k+1}}{I^{n-k+1}}.
	\end{align*}
	Let $\{M_j\}$ be a set of monomials in $x_1,\dots,x_d$ of degree $n-k$. Let $\sum a_jM_j\in I^{n-k+1}$ with $a_j\in I_{k+1}$. Suppose $\sum a_jM_j=\sum b_iM_i$ where $b_i\in I$. Since $x_1,\dots,x_d$ is a regular sequence, $a_j-b_j\in I$ and hence $a_j\in I$. The assumption implies that $I\cap I_{k+1}=I_kI$. This proves the isomorphism in \eqref{sec1:lem1:eq1}. 
	
	Finally, as $I^{n-k}/I^{n-k+1}$ is a free $R/I$-module for all $n\ge k$, \eqref{sec1:lem1:eq1} gives us 
	\[ \ell\left(\frac{I^{n-k}I_{k+1}}{I^{n-(k-1)}I_k}\right) = \binom{n+d-k-1}{d-1} \ell\left(\frac{I_{k+1}}{II_k}\right). \]
\end{proof}

The following theorem generalizes \cite[Theorem 10]{itoh88}.
\begin{Theorem}\label{polyfilt}
	Let $(R,\m)$ be  $d$-dimensional Cohen-Macaulay local ring and $I$ be an ideal generated by a system of parameters. Let $\f=\{I_n\}$ be an $I$-admissible filtration satisfying the condition $HI_0$. Then $I_{n+1} = I_2I^{n-1}$ for all $n \geq 1$ if and only if for all $n\geq 0,$
	\[ \ell\left(\frac{R}{I_{n+1}}\right)=e_0(\f)\binom{n+d}{d}- e_1(\f)\binom{n+d-1}{d-1} + e_2(\f)\binom{n+d-2}{d-2} \]
	where 
	\[ e_0(\f)=\ell\left(\frac{R}{I}\right), \, e_1(\f)= \left[ \ell\left(\frac{I_1}{I}\right) + \ell\left(\frac{I_2}{I I_1} \right) \right], \, e_2(\f)=\ell\left(\frac{I_2}{II_1}\right) \]
	and $e_i(\f)=0$, for all $3\le i \le d$.
\end{Theorem}

\begin{proof}
	First, notice that $I^nI_1 \subseteq I_{n+1}$. So we have
	\begin{align*}
	\ell\left(\frac{R}{I_{n+1}}\right)
	&= \ell\left(\frac{R}{I^nI_1}\right) - \ell\left(\frac{I_{n+1}}{I^nI_1}\right)\\
	&= \ell\left(\frac{R}{I^n}\right) + \ell\left(\frac{I^n}{I^nI_1}\right) - \ell\left(\frac{I_{n+1}}{I^nI_1}\right)\\
	&= \ell\left(\frac{R}{I}\right){n+d-1\choose d} + \ell\left(\frac{I^n}{I^nI_1}\right) - \ell\left(\frac{I_{n+1}}{I^nI_1}\right).
	\end{align*}
	Since $I^n/I^{n+1}\otimes_R R/I_1 \simeq I^n/I^nI_1$ and $I^n/I^{n+1}$ is a free $R/I$-module, we have 
	\begin{align*}
	\ell\left(\frac{R}{I_{n+1}}\right)
	&= \ell\left(\frac{R}{I}\right){n+d-1\choose d} + {n+d-1\choose d-1}\ell\left(\frac{R}{I_1}\right) - \ell\left(\frac{I_{n+1}}{I^nI_1}\right)\\
	&= \ell\left(\frac{R}{I}\right){n+d-1\choose d} + {n+d-1\choose d-1}\left[ \ell\left(\frac{R}{I}\right) - \ell\left(\frac{I_1}{I}\right) \right] - \ell\left(\frac{I_{n+1}}{I^nI_1}\right)\\
	&= \ell\left(\frac{R}{I}\right){n+d\choose d} - {n+d-1\choose d-1}\ell\left(\frac{I_1}{I}\right) - \ell\left(\frac{I_{n+1}}{I^nI_1}\right)\\
	&= \ell\left(\frac{R}{I}\right){n+d\choose d} - {n+d-1\choose d-1}\ell\left(\frac{I_1}{I}\right) - \ell\left(\frac{I^{n-1}I_2}{I^nI_1}\right) - \ell\left(\frac{I_{n+1}}{I^{n-1}I_2}\right).
	\end{align*}
	As the filtration $\f$ satisfies the condition $HI_0$, using Lemma \ref{iso} it follows that for $n \ge 1$,
	\[ \frac{I^{n-1}I_{2}}{I^{n}I_1}\cong \frac{I^{n-1}}{I^{n}}\bigotimes_R\frac{I_{2}}{II_1}. \]
	As $I^{n-1}/I^n$ is a free $R/I$-module, we conclude that  \[ \ell\left(\frac{I^{n-1}I_2}{I^nI_1}\right) = \ell\left(\frac{I_2}{II_1}\right){n+d-2\choose d-1} = \ell\left(\frac{I_2}{II_1}\right) \left[{n+d-1\choose d-1}-{n+d-2\choose d-2}\right]. \]
	Thus for $n \ge 1$,
	\begin{align*}
	\ell\left(\frac{R}{I_{n+1}}\right)
	&= \ell\left(\frac{R}{I}\right)\binom{n+d}{d}-\left[ \ell\left(\frac{I_1}{I}\right) + \ell\left(\frac{I_2}{I I_1} \right) \right] \binom{n+d-1}{d-1}\\
	& + \ell\left(\frac{I_2}{II_1}\right)\binom{n+d-2}{d-2}- \ell\left(\frac{I_{n+1}}{I^{n-1}I_2}\right).
	\end{align*}
\end{proof} 
We now find a formula of the Hilbert-Samuel polynomial of an $I$-admissible filtration $\f$ which satisfies the condition $HI_p$, for all $p\le r(\f)-2$.
\begin{Lemma}\label{binomialEquality}
	The following equality holds for all $d,k \ge 1$, 
	\begin{align*}
	\binom{n+d-k-1}{d-1}=\sum_{j=1}^{k+1} (-1)^{j-1}\binom{k}{j-1}\binom{n+d-j}{d-j}.
	\end{align*}
\end{Lemma}
\begin{proof}
	We derive the formula by applying  induction on the difference $(d-1)-(d-k-1) =k$. The result is clearly true for $k=1$, establishing the base case of induction. Suppose the result is true for $k$. We prove the result for $k+1$. Observe that
	\begin{align*}	\binom{n+d-(k+1)-1}{d-1}&=\binom{n+d-k-1}{d-1}-\binom{n+d-(k+1)-1}{d-2}\\
	&=\binom{n+d-k-1}{d-1}-\binom{n+(d-1)-k-1}{(d-1)-1}.
	\end{align*}
	Using the induction hypothesis, we have
	\begin{align*}
	\binom{n+d-k-2}{d-1}
	&=\sum_{i=1}^{k+1} (-1)^{i-1}\binom{k}{i-1}\binom{n+d-i}{d-i}\\
	&\hskip5cm -\sum_{j=1}^{k+1} (-1)^{j-1}\binom{k}{j-1}\binom{n+(d-1)-j}{(d-1)-j}\\
	&=\sum_{i=1}^{k+1} (-1)^{i-1}\binom{k}{i-1}\binom{n+d-i}{d-i}-\sum_{j=2}^{k+2} (-1)^{j-2}\binom{k}{j-2}\binom{n+d-j}{d-j}.
	\end{align*}
	Rearranging the terms,
	\begin{align*}
	\binom{n+d-k-2}{d-1}
	&=\binom{n+d-1}{d-1}+\sum_{i=2}^{k+1} (-1)^{i-1}\left[\binom{k}{i-1}+\binom{k}{i-2}\right]\binom{n+d-i}{d-i}\\
	&\hskip5cm +(-1)^{k+1}\binom{n+d-2-k}{d-2-k}\\
	&=\binom{n+d-1}{d-1}+\sum_{i=2}^{k+1} (-1)^{i-1}\binom{k+1}{i-1}\binom{n+d-i}{d-i}\\
	&\hskip5cm +(-1)^{k+1}\binom{n+d-2-k}{d-2-k}\\
	&=\sum_{i=1}^{k+2} (-1)^{i-1}\binom{k+1}{i-1}\binom{n+d-i}{d-i}.
	\end{align*}
	Thus the result holds by induction. 
\end{proof}
\begin{Theorem}\label{HSpoly}
	Let $R$ be a $d$-dimensional Cohen-Macaulay Noetherian local ring and $I$ be an ideal generated by an $R$-regular sequence. Let $\f=\{I_n\}$ be an $I$-admissible filtration and $r(\f)$ denote the reduction number of $\f$. If $\f$ satisfies the condition $HI_p$, for $p\leq r=r(\mathcal{F})-2$, then for all $i=1,\dots, d$,
	\begin{align*}
	e_i(\f)=\sum_{k=i-1}^{\infty}\binom{k}{i-1}\ell\left(\frac{I_{k+1}}{II_k}\right).
	\end{align*}
\end{Theorem}	
\begin{proof}
	Recall that for large $n$,
	\begin{align*}
	\ell\left(\frac{R}{I_{n+1}}\right)=e_0(\f)\binom{n+d}{d}-e_1(\f)\binom{n+d-1}{d-1}+\cdots+(-1)^de_d(\f).
	\end{align*}
Let $r=r(\f)$. For sufficiently large $n$, as in the proof of Theorem \ref{polyfilt}, we can write
	
	\begin{align*}
	\ell\left(\frac{R}{I_{n+1}}\right)
	&=\ell\left(\frac{R}{I}\right)\binom{n+d}{d}-\ell\left(\frac{I_1}{I}\right)\binom{n+d-1}{d-1}-\ell\left(\frac{I_{n+1}}{I^nI_1}\right)\\
	&=\ell\left(\frac{R}{I}\right)\binom{n+d}{d}-\ell\left(\frac{I_1}{I}\right)\binom{n+d-1}{d-1}-\sum_{k=1}^{r-1}\ell\left(\frac{I^{n-k}I_{k+1}}{I^{n-k+1}I_k}\right)+\ell\left(\frac{I_{n+1}}{I^{n+1-r}I_r}\right).
	\end{align*}
	Since $\f$ satisfies $HI_p$ for $p\leq r-2$, we can use Lemma \ref{iso} to conclude that $\ell\left(\dfrac{I^{n-k}I_{k+1}}{I^{n-k+1}I_k}\right)=\binom{n+d-k-1}{d-1}\ell\left(\dfrac{I_{k+1}}{II_k}\right)$, for all $1 \le k \le r-1$. Also, since $r$ is the reduction number of $\f$, we have $\ell\left(\dfrac{I_{n+1}}{I^{n+1-r}I_r}\right)=0$. Thus
	\begin{align*}
	\ell\left(\frac{R}{I_{n+1}}\right)&=\ell\left(\frac{R}{I}\right)\binom{n+d}{d}-\ell\left(\frac{I_1}{I}\right)\binom{n+d-1}{d-1}-\sum_{k=1}^{r-1}\binom{n+d-k-1}{d-1}\ell\left(\frac{I_{k+1}}{II_k}\right)\\
	&=\ell\left(\frac{R}{I}\right)\binom{n+d}{d}-\ell\left(\frac{I_1}{I}\right)\binom{n+d-1}{d-1}\\
	&\qquad\qquad\qquad\qquad\qquad -\sum_{k=1}^{r-1}\left[\sum_{j=1}^{k+1} (-1)^{j-1}\binom{k}{j-1}\binom{n+d-j}{d-j}\right] \ell\left(\frac{I_{k+1}}{II_k}\right).
	\end{align*}
where the last equality is a consequence of Lemma \ref{binomialEquality}. 
	Now equate  the coefficient of $\binom{n+d-i}{d-i}$ on both sides to get  
	\[ e_i(\f)=\sum_{k=i-1}^{\infty}\binom{k}{i-1}\ell\left(\frac{I_{k+1}}{II_k}\right). \]

\end{proof}

\begin{Remark} {\rm
		In \cite[Corollary 2.11]{Huckaba1}, Huckaba derived the above formulas when $I$ is an $\m$-primary ideal and $\f$ is the $I$-adic filtration, provided that the condition $\depth \gr_I(R)\geq d-1$ is satisfied. So it is natural to question the strength of the statement $I_{n+1}\cap I^{n-p}= I_{p+1}I^{n-p}$ for all $0\leq p\leq r(\f)-2$, when compared to $\depth \gr_{\f}(R)\geq d-1$.
}\end{Remark}

\begin{Theorem}
	If $\f$ satisfies the condition $HI_p$ for all $p$, then the associated graded ring $\gr_{\f}(R)$ is Cohen-Macaulay.
\end{Theorem}
\begin{proof}
	We have $I_{n+1}\cap I^{n-p}= I_{p+1}I^{n-p}$ for all $p \le n$. Put $p=n-1$ to get for all $n\geq 0$,
	\[ I_{n+1}\cap I=I_n I. \] 
	Hence, by Valabrega-Valla (\cite[Corollary 2.7]{VV} or \cite[Lemma 9]{itoh88}), we get that $\gr_{\f}(R)$ is Cohen-Macaulay.  
\end{proof}
It follows that the condition that a filtration $\f$ satisfies the condition $HI_p$, for all $p$, is a stronger condition than $\depth \gr_{\f}(R)\geq d-1$. Thus, the generalized intersection theorems only help us to give a different proof of Huckaba's result.

\section{The Tight Hilbert Polynomial}

Let $(R, \m)$ be a $d$-dimensional analytically unramified local ring with positive prime characteristic $p$ and let $I$ be an $\m$-primary ideal. We prove that for large $n$, $H_I^*(n) = \ell(R/(I^n)^*)$ is a polynomial, $P_I^*(n)$, of degree $d$ and with coefficients in $\Q$. We write 
\[ P_I^*(x) = e_0^*(I){x+d-1 \choose d} - e_1^*(I){x+d-2 \choose d-1}+\cdots + (-1)^d e_d^*(I) \]
where $e_i^*(I)\in \Z$.

\begin{Theorem}
	Let $R$ be a d-dimensional analytically unramified local ring with positive prime characteristic $p$ and $I$ be an $\m$-primary ideal. Let $\g=\{(I^n)^*\}_{n\in \Z}$. Then\\
	{\rm(1)} $\g$ is an I-filtration.\\
	{\rm(2)} $R(\g) = \oplus_{n\in \Z} (I^n)^*t^n$ is a finite $R(I)$-module. \\
	{\rm(3)} $H_I^*(n) = \ell(R/(I^n)^*)$ is a polynomial function of degree $d$, given by $P_I^*(x) \in \Q[x]$. Moreover, $P_I^*(x)$ and $P_I(x)$ have the same leading coefficient, $e_0(I)$.	
\end{Theorem}

\begin{proof}
	Clearly, $(I^{n+1})^* \subseteq (I^n)^*$ and $I^n \subseteq (I^n)^*$, for all $n$. 
	Observe that for any pair of ideals $I,J$ of $R$, we have $I^*J^* \subseteq (I^*J^*)^*=(IJ)^*.$ Thus, $(I^m)^* \cdot (I^n)^* \subseteq (I^{m+n})^*$, for all $m,n \in \Z.$ 
	Therefore, $\g$ is an $I$-filtration, proving (1). 
	
	Let $R(\f) = \oplus_{n\in\Z}\ov{I^n}t^n$ and $R(\g) = \oplus_{n\in\Z}(I^n)^*t^n$. We have 
	\[ R(I)\hookrightarrow R(\g) \hookrightarrow R(\f), \]
	which is an inclusion of graded rings. As $R(\f)$ is Noetherian and is a finite $R(I)$-module and as $R(\g)$ is a Noetherian submodule of $R(\f)$, we get $R(\g)$ is a finite $R(I)$-module. Thus, there exists $k\in\N$ such that $(I^n)^* \subseteq I^{n-k}$, for all $n$. Hence, $\g$ is an $I$-admissible filtration. This proves (2).
	
	Using Proposition \ref{poly}, we can conclude that $H^*_I(n)=\ell(R/(I^n)^*)$ is a polynomial, $P^*_I(x) \in \Q[x]$. We write 
	\[ P_I^*(n) = e_0^*(I){n+d-1 \choose d} - e_1^*(I){n+d-2 \choose d-1}+\cdots + (-1)^d e_d^*(I) \]
	where $e_i^*(I)\in \Z$. Moreover, $P^*_I(n)$ and $P_I(n)$ have the same leading coefficient.
\end{proof}

We now prove that in characteristic $p$, the Huneke-Itoh intersection theorem is indeed true for the  tight closure filtration of an ideal generated by a regular sequence. The following proposition was assigned as an exercise in a course offered by Huneke at Purdue University in Spring 1987. We supply a proof due to lack of a suitable reference.

For $\alpha=(\alpha_1,\alpha_2,\ldots,\alpha_d) \in \N^d$, let $|\alpha|=\alpha_1+\alpha_2+\cdots+\alpha_d$ and $x^\alpha=x_1^{\alpha_1}x_2^{\alpha_2}\cdots x_d^{\alpha_d}.$

\begin{Proposition}\label{intersectTight}
	Let $R$ be a Noetherian ring and I be an ideal generated by a regular sequence. Then for all $n \ge 0$, 
	\[ I^n \cap {(I^{n+1})^*}=I^n {I^*}. \]
\end{Proposition}

\begin{proof} Let $I$ be generated by a regular sequence $x_1, x_2,\ldots, x_d.$ The inclusion $I^n {I^*} \subset I^n \cap {(I^{n+1})^*}$ is clear since $I^*J^* \subset (I^*J^*)^* = (IJ)^*$, for any ideals $I,J$ of $R.$ For the reverse inclusion, let $ a \in I^n \cap {(I^{n+1})^*}.$ Then there exists a $c \in R^o$ such that $c a^q \in (I^{n+1})^{[q]}$, for all large $q.$ 
	
	We write $a=\sum_{|\alpha|=n} y_\alpha x^\alpha$, where $y_\alpha \in R.$ Then 
	\[ ca^q= \sum_{|\alpha|=n} cy_\alpha^q x_1^{q\alpha_1}\cdots x_d^{q\alpha_d} = \sum_{\beta\in \N^{d}, \; |\beta|=n+1} z_\beta x_1^{q\beta_1}\cdots x_d^{q\beta_d} \]
	for some $z_\beta \in R.$ Using the fact that $x_1^q, \ldots,x_d^q$ is a regular sequence, we conclude that $cy^q_\alpha \in (I)^{[q]}$, for all large $q.$ This implies that $y_\alpha \in I^*$ and hence $a \in I^nI^*.$	
\end{proof}

\begin{Corollary}[Watanabe]\label{Watanabe}
	Let $R$ be a Noetherian  ring  having  positive prime characteristic $p$ and $I$ be an ideal generated by a regular sequence. If $I$ is tightly closed, then so is $I^n$, for all $n\ge 1$. 
\end{Corollary}
\begin{proof}
	We prove the result by induction on $n$. Clearly, the result holds for $n=1$. Suppose that $I, I^2,\dots, I^{n-1}$ are tightly closed. Let $x\in (I^n)^*.$ Since $ (I^n)^*\subset (I^{n-1})^*=I^{n-1},$ using the above proposition it follows that
	$x\in (I^n)^*\cap I^{n-1}=I^*I^{n-1}=I^n.$ Hence $I^n$ is also tightly closed.  
\end{proof}

We characterize $F$-rationality of a ring $R$ with positive prime characteristic $p$ in terms of vanishing of $e_1^*(I)$, under suitable conditions. Let $q=p^e,$ for some $e>0$.

\begin{Theorem}
	Let $(R,\m)$ be a $d$-dimensional analytically unramified Cohen-Macaulay local ring with positive prime characteristic $p$. Then $R$ is $F$-rational if and only if $e_1^*(I)=0,$ for some  ideal $I$ generated by a system of parameters.
\end{Theorem}
\begin{proof}
	Let $R$ be $F$-rational and $I$ be generated by a system of parameters. Then $I$ is tightly closed and so are all the powers of $I$ by Proposition \ref{Watanabe}. Thus, \[\ell(R/(I^n)^*)=\ell(R/I^n)=e_0(I) {n+d-1\choose d}. \] 
	Hence, $e_1^*(I)=0$. Conversely, let $I$ be an ideal generated by a system of parameters  of $R$ such that $e_1^*(I)=0$. We need to show that $R$ is $F$-rational. Using \cite[Corollary 4.9]{huckabaMarley}, we know that $e_1^*(I)\ge e_0^*(I)-\ell(R/I^*)$. This implies that $e_0(I) = \ell(R/I)\le \ell(R/I^*)$. But as $I\subseteq I^*$, we get $I=I^*$. Since $R$ is Cohen-Macaulay, we conclude that $R$ is $F$-rational.
\end{proof}

The following proposition can now be proved as a consequence of Theorem \ref{polyfilt} and Proposition \ref{intersectTight} in characteristic $p>0$.

\begin{Proposition} \label{itohTC}
	Let $(R,\m)$ be  a $d$-dimensional Cohen-Macaulay analytically unramified local ring with positive prime characteristic $p$. Let $I$ be an ideal generated by a system of parameters in $R$. Then $(I^{n+1})^* = (I^2)^*I^{n-1}$ for all $n \geq 1$ if and only if for all $n \ge 0$,
	\[ \ell\left(\frac{R}{(I^{n+1})^*}\right)=e_0^*(I)\binom{n+d}{d}- e_1^*(I)\binom{n+d-1}{d-1} + e_2^*(I)\binom{n+d-2}{d-2} \]
	where 
	\[ e_0^*(I)=e_0(I)=\ell\left(\frac{R}{I}\right), \, e_1^*(I)= \left[ \ell\left(\frac{I^*}{I}\right) + \ell\left(\frac{(I^2)^*}{I I^*} \right) \right], \, e_2^*(I)=\ell\left(\frac{(I^2)^*}{II^*}\right) \]
	and $e_i^*(I)=0$, for all $3\le i \le d$.
\end{Proposition}

\section{The tight Hilbert polynomial in diagonal hypersurface rings}

In this section, we calculate the tight Hilbert polynomial of a parameter ideal in diagonal hypersurface  rings. Let $R=\F[[X,Y,Z]]/(X^N+Y^N+Z^N)$, where $N \ge 2$, $\F$ is an infinite field of characteristic $p$ and $p \nmid N$. Let $x,y,z$ denote the images of $X,Y,Z$ respectively in $R$. Let $I=(y,z)$ and $\m=(x,y,z)$. We shall find $(I^k)^*$ for all $k$ and show that for $N \ge 3$, $r^*(I)=N-2$. We recall a few definitions first. Let $q=p^e$, for some $e \in \mathbb{N}.$

\begin{Definition}	[{\rm \cite[1.4.8]{hochstertight0}}] 
	{\rm If $R$ is a finitely generated algebra over a reduced Noetherian domain $A$, we say that $R$ is \emph{generically smooth} over $A$ if there exists $a \in A^\circ$ such that $R_a$ is smooth over $A_a$. 
}\end{Definition}

\begin{Definition}{\rm
	We say that a Noetherian local ring $(A,\m)$ is \emph{formally equidimensional} if its $\m$-adic completion $\hat{A}$ is equidimensional.
}\end{Definition}

\begin{Definition}{\rm
	Let $R$ be a Noetherian ring. An element $c \in R^\circ$ is called a test element if for all ideals $I$ of $R$ and all $x \in I^*$, we have $cx^q \in I^{[q]}$, for all $q$.
}\end{Definition}

\begin{Remark}\label{HypSurfaceSec:rmk1}{\rm
	(i) Since $p \nmid N$, $X^N+Y^N+Z^N$ is an irreducible element in $\F[[X,Y,Z]]$. Indeed, set $T=Y/Z$. Then $Y^N+Z^N=Z^N(T^N+1)$. The polynomial $T^N+1$ has a non-zero derivative and hence is separable over $\ov{\F}$. This implies that $Y^N+Z^N$ is a product of $N$ distinct linear forms. Using Eisenstein's criterion, $X^N+Y^N+Z^N$ is irreducible over $\ov{\F}$ and hence over $\F$. Thus, $R$ is a domain.
		
	(ii) Observe that $\gr_{\m}(R)$ is isomorphic to $\F[X,Y,Z]/(X^N+Y^N+Z^N)$ and therefore, it is reduced. This implies that $\ov{\m^k} = \m^k$, for all $k$. As $I$ is a reduction of $\m$, we get $\overline{I^k}=\m^k$, for all $k$.
}\end{Remark}

M. Hochster discovered  a remarkable result  on test elements which is often used for calculation of tight closure. Let $A \subseteq R$ be a module-finite extension, where $A$ is a Noetherian domain, $R$ is a torsion-free $A$-module and the extension is generically smooth. Write $R \simeq A[x_1,\dots,x_n]/P$. Then $\J=\J(R/A)$ is the ideal generated in $R$ by the images of all the Jacobian determinants $\partial(g_1,\dots,g_n)/\partial(x_1,\dots,x_n)$, for $n$-tuples $g_1,\dots,g_n$ of elements of $P$. Moreover, to generate $\J(R/A)$ it suffices to take all the $n$-tuples of $g_i$ from a fixed set of generators of $P$.

\begin{Theorem}[{\cite[Corollary 8.2]{HochLeus}}]\label{testelements}
	Let $R$ be a domain that is module-finite over a regular domain $A$ of characteristic $p>0$, such that the corresponding extension of fraction fields is separable. Then every non-zero element of $\J$ is a test element, where $\J=\J(R/A)$ is the relative Jacobian ideal.
\end{Theorem}

Let $A=\F[X,Y]$. Then $R \simeq A[Z]/(X^N+Y^N+Z^N)$. As the extension $A\subseteq R$ satisfies all the assumptions of the above theorem, $z^{N-1} \in \J(R/A)$ is a test element. 

\begin{Lemma}\label{initialideal}
	Let $A=\F[X,Y,Z]$ be a polynomial ring. Let $k \in \N$ and $q=p^e$, for some $e$. Let $J=(X^N+Y^N+Z^N,(Y^q,Z^q)^k)$ be an ideal in $A$. If $>$ denotes the graded reverse lex ordering with $X>Y>Z$, then $\ini_{>}(J)=(X^N,(Y^q,Z^q)^k)$. 
\end{Lemma}
\begin{proof}
	It is sufficient to show that $G=\{X^N+Y^N+Z^N, Y^{qr}Z^{q(k-r)}, r=0,1,\dots,k \}$ forms a Gr\"{o}bner basis of $J$. For $r \in \{0,1,\dots,k \}$, consider the $S$-polynomial of polynomials $f=X^N+Y^N+Z^N$ and $g=Y^{qr}Z^{q(k-r)}$.
	\begin{align*}
	S(f, g) &= \frac{X^NY^{qr}Z^{q(k-r)}}{X^N}  (X^N+Y^N+Z^N) - \frac{X^NY^{qr}Z^{q(k-r)}}{Y^{qr}Z^{q(k-r)}}  Y^{qr}Z^{q(k-r)}\\
	&= Y^{qr}Z^{q(k-r)}(Y^N+Z^N).
	\end{align*}
	Using division algorithm, one finds that $\ov{S(X^N+Y^N+Z^N, Y^{qr}Z^{q(k-r)})}^G=0$. Now, for some $r_1,r_2 \in \{0,1,\dots,k \}$, $r_1 \not= r_2$, consider the $S$-polynomial $S(Y^{qr_1}Z^{q(k-r_1)}, Y^{qr_2}Z^{q(k-r_2)})$. As $S$-polynomial of monomials is zero, it follows that $G$ forms a Gr\"{o}bner basis of $J$ and hence $\ini_{>}(J)=(X^N,(Y^q,Z^q)^k)$.	
\end{proof}

\begin{Proposition}
	If $N=2$, then $R$ is $F$-rational. 
\end{Proposition}
\begin{proof}
	As $R$ is Cohen-Macaulay, it is sufficient to show that $I=(y,z)$ is tightly closed. As $\ell(\m/I)=1$, it follows that $I$ is tightly closed if and only if $x \notin I^*$. As $z$ is a test element, we get $x \notin I^*$ if and only if $zx^q \notin I^{[q]} = (y^q,z^q)$, for some $q$. Let $J_q=(X^2+Y^2+Z^2,Y^q,Z^q)$ be an ideal in the polynomial ring $\F[X,Y,Z]$. We show that for some $q$, $X^qZ \notin J_q$.
	
	Write $q=2u+1$ and $q>2$. Let $>$ denote the graded reverse lex order with $X>Y>Z$ in the ring $\F[X,Y,Z]$. Using Lemma \ref{initialideal}, it follows that $\ini_{>}(J_q)=(X^2,Y^q,Z^q)$. Then 
	\[ X^qZ=X^{2u+1}Z\equiv (-1)^uX(Y^2+Z^2)^uZ \mod J_q. \]
	This implies that $X^qZ - (-1)^uX(Y^2+Z^2)^uZ \in J_q$. Thus, $X^qZ \in J_q$ if and only if $(-1)^uX(Y^2+Z^2)^uZ \in J_q$. But 
	\[ \ini_{>} ((-1)^uX(Y^2+Z^2)^uZ)=XY^{2u}Z = XY^{q-1}Z. \]
	Since $q>2$,  $\ini_{>}((-1)^uX(Y^2+Z^2)^uZ)=XY^{q-1}Z \not\in \ini_{>}(J_q)$. Thus $ZX^q \notin J_q$ and hence $ZX^q \notin J_q\F[[X,Y,Z]]$. This implies that $zx^q\not\in I^{[q]}$. Therefore, $x\not\in (I)^*$.
\end{proof}

\begin{Remark}{\rm
		If $N=2$, then using Corollary \ref{Watanabe}, $(I^k)^* = I^k$ for all $k \ge 1$. 
}\end{Remark}

\begin{Proposition}
	{\rm (i) }$(I^k)^*=\m^{k+1}+I^k$, for all $k \ge 1$ and
	{\rm (ii)} $r^*(I)=N-2$.
\end{Proposition}
\begin{proof}
	$(i)$: Observe that, using \cite[Lemma 3.1]{sullivant}, the tight closure of a monomial ideal in a hypersurface ring is a monomial ideal. Notice that $(I^k)^*\subseteq\overline{I^k}=\m^k$, for all $k \ge 1$ (Remark \ref{HypSurfaceSec:rmk1}(ii)). 
	We prove the following statements.\\
	(1) $\m^{k+1}\subseteq (I^k)^*$.\\
	(2) $x^ay^bz^c\not\in (I^k)^*$ for $a\geq 1$ and $a+b+c=k$.\\ 
	(3) $x^ay^bz^c\not\in (I^k)^*$ for $a+b+c < k$.
	
	Suppose the above statements have been proved. Then the conclusion $I^k + \m^{k+1} \subseteq (I^k)^*$ of $(1)$ is obvious. It also shows that any monomial in $\m^k$ whose exponent of $x$ is zero is in $(I^k)^*$. On the other hand, the statement $(2)$ shows that any monomial of degree $k$ in $\m^k$ whose exponent of $x$ is 1 or greater cannot be in $(I^k)^*$. The last part proves that the remaining monomials, i.e., monomials of degree less than $k$ in $R$ are not in $I^*$. Thus, the above statements give a proof of $(i)$.
	
	We now proceed to give a proof of the statements (1),(2) and (3) above. \\
	(1) As $I^{k+1}\subseteq (I^k)^*$, we show that $x^ay^bz^c\in (I^k)^*$, where $a+b+c=k+1$ and $a \ge 1$. We show that $z^{N-1} (x^ay^bz^c)^q\in (I^k)^{[q]}=(y^q,z^q)^k$, where $q=p^e$ for large $e$. Choose $q>N$. Let $aq=Nu+i$ where $i\in \{1,\dots,N-1\}$. Then 
	\begin{align} 
	z^{N-1}(x^ay^bz^c)^q
	&=y^{bq}z^{cq+N-1}x^{aq}=y^{bq}z^{cq+N-1}x^{Nu+i}\nonumber\\
	&=(-1)^uy^{bq}z^{cq+N-1}x^i(y^N+z^N)^u\nonumber\\
	&=(-1)^uy^{bq}z^{cq+N-1}x^i\sum_{j=0}^u{u\choose j}y^{Nj}z^{N(u-j)}. \label{hypersurfaceSectioneq1}
	\end{align}
	
	We first re-write the exponents of $y$ and $z$ in \eqref{hypersurfaceSectioneq1} in the form, namely $r_1q+s_1$ and $r_2q+s_2$ where $s_1,s_2\geq 0$ respectively. Now to check if $z^{N-1}(x^ay^bz^c)^q\in (y^q,z^q)^k$, it is enough to check if $r_1+r_2\geq k$. The exponents of $y$ in \eqref{hypersurfaceSectioneq1} are $Nj+bq$ and the exponents of $z$ in \eqref{hypersurfaceSectioneq1} are $N(u-j)+cq+N-1$, for all $j=0,1,\dots,u$.
	
	Suppose $Nj=dq+r$ where $0<r<q$. Then the exponent of $y$ is $(b+d)q+r$ and the exponent of $z$ is
	\begin{align*}
	N(u-j)+N-1+cq&=Nu-Nj+N-1+cq\\
	&=aq-i-dq-r+N-1+cq\\
	&=(a+c-d-1)q+(q-(i+r-N+1)).
	\end{align*}
	Since $i<N, r<q$, we have $i+r+1\leq N+q$ or $i+r-N+1\leq q$. If $0\leq i+r-N+1\leq q$, then the sum of the multiples of $q$ appearing in the exponent of $y$ and $z$ is 
	\[ (b+d)q+(a+c-d-1)q=(a+b+c-1)q=kq. \]
	
	Also, if $i+r-N+1<0$, then $i+r-N+1=-h$ for some $h>0$. Thus the exponent of $z$ is 
	\[ (a+c-d-1)q+(q-(i+r-N+1))=(a+c-d-1)q+q+h=(a+c-d)q+h. \] 
	Now the sum of multiples of $q$ appearing in the the exponent of $y$ and $z$ is $(b+d)q+(a+c-d)q=(k+1)q$. Therefore, $z^{N-1} (x^ay^bz^c)^q\in (I^k)^{[q]}$ and hence $x^ay^bz^c\in (I^k)^*$, where $a+b+c=k+1$ and $a\ge 1$.
	
	(2) Now we show that $x^ay^bz^c\not\in (I^k)^*$, if $a\geq 1$ and $a+b+c=k$. Observe that $x^ay^bz^c\notin (I^k)^*$ if and only if $z^{N-1}(x^ay^bz^c)^q \notin (I^k)^{[q]} = (y^q,z^q)^k$, for $q=p^e$ and $e$ large. Let $q>N$ and $aq=Nu+i$ where $i\in\{1,\dots,N-1\}$. 
	
	Let $>$ denote the graded reverse lex order with $X>Y>Z$ in $A=\F[X,Y,Z]$. Let $J=(X^N+Y^N+Z^N,(Y^q,Z^q)^k)$ be an ideal in $A$. Using Lemma \ref{initialideal}, it follows that $\ini_{>}(J)=(X^N,(Y^q,Z^q)^k)$. Now
	\begin{align*}
	Z^{N-1}(X^{aq}Y^{bq}Z^{cq})=X^{Nu}X^iY^{bq}Z^{cq+N-1} \equiv (-1)^uX^i(Y^N+Z^N)^uY^{bq}Z^{cq+N-1} \mod J.
	\end{align*}
	Let $f=(-1)^uX^i(Y^N+Z^N)^uY^{bq}Z^{cq+N-1}$. It follows that $Z^{N-1}(X^{aq}Y^{bq}Z^{cq}) - f \in J$. Therefore, $Z^{N-1}(X^{aq}Y^{bq}Z^{cq}) \in J$ if and only if $f \in J$. Observe that  \[ \ini_{>}(f)=\ini_{>}((-1)^uX^i(Y^N+Z^N)^uY^{bq}Z^{cq+N-1})=X^iY^{bq+Nu}Z^{cq+N-1}. \]
	The exponent of $Y$ in $\ini_{>}(f)$ is
	\begin{align*}
	bq+Nu=bq+aq-i=(b+a)q-i=(k-c)q-i=(k-c-1)q+(q-i).
	\end{align*}
	The exponent of $Z$ in $\ini_{>}(f)$ is $cq+N-1$. The multiples of $q$ appearing in the exponents of  $Y$ and $Z$ is $(k-c-1)q$ and $cq$ respectively. Thus the sum of the multiples of $q$ appearing in the exponents of  $Y$ and $Z$ is $(k-c-1)q+cq=(k-1)q$. Since $N<q$, $\ini_{>}(f)\not\in \ini_{>}(J)$. Thus $Z^{N-1}X^{aq}Y^{bq}Z^{cq}\not\in J$ and hence $Z^{N-1}X^{aq}Y^{bq}Z^{cq}\not\in J\F[[X,Y,Z]]$. This implies that $z^{N-1}x^{aq}y^{bq}z^{cq}\not\in (I^k)^{[q]}$ or $x^{aq}y^{bq}z^{cq}\not\in (I^k)^*$.
	
	(3) Now we show that $x^ay^bz^c\not\in (I^k)^*$, for $a+b+c<k$. Observe that $x^ay^bz^c\notin (I^k)^*$ if and only if $z^{N-1}(x^ay^bz^c)^q \notin (I^k)^{[q]} = (y^q,z^q)^k$, for $q=p^e$ and $e$ large. Let $q>N$ and $aq=Nu+i$ where $i\in\{1,\dots,N-1\}$. 
	
	Let $>$ denote the graded reverse lex order with $X>Y>Z$ in $A=\F[X,Y,Z]$. Let $J=(X^N+Y^N+Z^N,(Y^q,Z^q)^k)$ be an ideal in $A$. Using Lemma \ref{initialideal}, it follows that $\ini_{>}(J)=(X^N,(Y^q,Z^q)^k)$. Using arguments as above, we get
	\begin{align*}
	Z^{N-1}(X^{aq}Y^{bq}Z^{cq}) \equiv (-1)^uX^i(Y^N+Z^N)^uY^{bq}Z^{cq+N-1} \mod J.
	\end{align*}
	Let $f=(-1)^uX^i(Y^N+Z^N)^uY^{bq}Z^{cq+N-1}$. It follows that $Z^{N-1}(X^{aq}Y^{bq}Z^{cq}) \in J$ if and only if $f \in J$. Observe that $\ini_{>}(f)=X^iY^{bq+Nu}Z^{cq+N-1}.$ The exponent of $Y$ in $\ini_{>}(f)$ is
	\begin{align*}
	bq+Nu=(b+a)q-i=(b+a-1)q+(q-i).
	\end{align*}
	The exponent of $Z$ in $\ini_{>}(f)$ is $cq+N-1$. The multiples of $q$ appearing in the exponents of $Y$ and $Z$ are $(b+a-1)q$ and $cq$ respectively. Thus the sum of the multiples of $q$ appearing in the exponents of  $Y$ and $Z$ is $(b+a-1)q+cq<kq$. Since $N<q$, $\ini_{>}(f)\not\in \ini_{>}(J)$. Thus $Z^{N-1}X^{aq}Y^{bq}Z^{cq}\not\in J$ and hence $Z^{N-1}X^{aq}Y^{bq}Z^{cq}\not\in J\F[[X,Y,Z]]$. This implies that $z^{N-1}x^{aq}y^{bq}z^{cq}\not\in (I^k)^{[q]}$ or $x^{aq}y^{bq}z^{cq}\not\in (I^k)^*$.
	
	
	\noindent $(ii)$: We now prove that $r^*(I)=N-2$, i.e., $(I^k)^* = I(I^{k-1})^*$, for all $k \ge N-1$. In particular, we need to show that $\m^{k+1}+I^k = I(\m^k+I^{k-1}) = \m^kI+I^k$, for all $k \ge N-1$. Hence, it is enough to show that $I\m^k = \m^{k+1}$, for all $k \ge N-1$. Observe that 
	\[ x^N=-(y^N+z^N) = -y(y^{N-1})-z(z^{N-1}) \in I\m^{N-1}. \]
	This implies that $I\m^{N-1} = \m^N$. As $I$ is a reduction of $\m$ and as $I\m^{N-1} = \m^N$, we get $(I^k)^* = I(I^{k-1})^*$, for all $k \ge N-1$.
	
	Therefore, it is sufficient to show that $(I^{N-2})^* \not= I(I^{N-3})^*$. Consider
	\begin{align*}
	I(I^{N-3})^* &= I(\m^{N-2}+I^{N-3})\\
	&=I( (x^{N-2})+(x^{N-3})I+\cdots+(x^3)I^{N-5}+(x^2)I^{N-4}+I^{N-3})\\
	&= (x^{N-2})I+(x^{N-3})I^2+\cdots+(x^3)I^{N-4}+(x^2)I^{N-3}+I^{N-2}.
	\end{align*}
	Note that $(I^{N-2})^* = \m^{N-1}+I^{N-2}=(x^{N-1})+(x^{N-2})I+\cdots +(x^2)I^{N-3}+I^{N-2}$. We claim that $x^{N-1} \notin I(I^{N-3})^* = I(\m^{N-2}+I^{N-3})$. Observe that the result holds if the claim is proved. Suppose 
	$x^{N-1} \in I(\m^{N-2}+I^{N-3})$. This implies that 
	\[ X^{N-1} \in (Y,Z)(X,Y,Z)^{N-2}+(Y,Z)^{N-2}+(X^N+Y^N+Z^N). \]
	As there are no pure powers of $X$ with degree less $N$ on the right hand side, we get a contradiction. Hence, the claim is proved.
\end{proof}

\begin{Example}\label{N=3}{\rm
	(${\bf N=3}$) Consider the ring $R=\F[[X,Y,Z]]/(X^3+Y^3+Z^3)$ of dimension 2 with char $\F \not= 3$ and ideal $I=(y,z)$. We observe that for $k \ge 1$,
	\[ (I^{k})^*=(x^2)I^{k-1}+I^k \]
	and $r^*(I)=1$. Using Proposition \ref{itohTC} it follows that for all $n\ge 1$
	\begin{align*}
	\ell\left(\frac{R}{(I^{n+1})^*}\right)
	&=\ell\left(\frac{R}{I}\right){{n+2}\choose 2} -\left[ \ell\left(\frac{I^*}{I}\right) + \ell\left(\frac{(I^2)^*}{II^*}\right) \right] {{n+1}\choose 1} + \ell\left(\frac{(I^2)^*}{II^*}\right).
	\end{align*}
	As $\ell(R/I)=3, \ell(I^*/I)=1$ and $\ell((I^2)^*/II^*)=0$, for $n\ge 1$ we have
	\[ \ell(R/(I^{n+1})^*)=3{{n+2}\choose 2}-(n+1). \]
}\end{Example} 

\begin{Example}\label{N=4}{\rm
	If $N=4,$ then the tight closures of powers of $I$ are given by:
	\begin{equation*}
	(I^k)^*=
	\begin{cases}
	(x^2,y,z)=(x^2)+I & \text{if $k=1$}\\
	(x^3)I^{k-2}+(x^2)I^{k-1}+I^k & \text{if $k\geq 2$}
	\end{cases}
	\end{equation*}
	and $r^*(I)= 2$. Using Proposition \ref{itohTC} for the ring $R=\F[[X,Y,Z]]/(X^4+Y^4+Z^4)$ of dimension 2 with char $\F \not= 2$ and ideal $I=(y,z)$, it follows that for all $n\ge 1$
	\begin{align*}
	\ell\left(\frac{R}{(I^{n+1})^*}\right)
	&=\ell\left(\frac{R}{I}\right){{n+2}\choose 2} -\left[ \ell\left(\frac{I^*}{I}\right) + \ell\left(\frac{(I^2)^*}{II^*}\right) \right] {{n+1}\choose 1} + \ell\left(\frac{(I^2)^*}{II^*}\right).
	\end{align*}
	As $\ell(R/I)=4, \ell(I^*/I)=2$ and $\ell((I^2)^*/II^*)=1$, for $n \ge 1$ we have
	\[ \ell(R/(I^{n+1})^*) = 4{{n+2}\choose 2}-3(n+1) + 1. \]
}\end{Example}

\section{The tight Hilbert polynomial in Stanley-Reisner rings}

In this section, we find the tight Hilbert polynomial of  a linear system of parameters in the Stanley-Reisner ring of a simplicial complex. We recall a few basic facts about Stanley-Reisner rings of simplicial complexes. Let $\Delta$ be a $(d-1)$-dimensional simplicial complex on $n$ vertices and $\kk$ be a field of a  positive prime characteristic $p.$ Let $I_{\Delta}$ be the ideal of $\Delta.$ Then $R= \kk[\Delta]=\kk[X_1,\dots ,X_n]/I_\Delta$ is a $d$-dimensional ring. Let $x_1,\dots,x_n$ denote the images of $X_1,\dots,X_n$ respectively in $R$, $\m=(x_1,\dots ,x_n)$ be the unique maximal homogeneous ideal of $\kk[\Delta]$ and $I$ be an ideal generated by a linear system of parameters.
The maximal faces under inclusion are called the facets of the simplicial complex. For any facet $F$ of $\Delta$, set $P_F=(x_i \mid i\notin F).$ Then $P_F$ is a minimal prime of $R$. Also, for any facet $F$ of $\Delta$, $\kk[F]=\kk[x_i \mid i \in F]$ is a polynomial ring and $\m_F=(x_i \mid i \in F)$ is its maximal homogeneous ideal. Note that $\kk[F]$ can also be seen as a residue class ring of $\kk[\Delta]$ in a natural way. Using \cite[Theorem 5.1.16]{brunsHerzog}, one can conclude that an ideal $I$ is generated by a linear system of parameters in $\kk[\Delta]$ if and only if for all facets $F$ of $\Delta$, $(I+P_F)/P_F \simeq \m_F$ in $\kk[F].$

Observe that as $R=\kk[\Delta]$ is a standard graded ring, it is isomorphic to the associated graded ring $\gr_\m(R)$. Hence $\gr_\m(R)$ is reduced. This implies that $\overline{\m^n} = (\m^n)^* = \m^n$, for all $n \ge 1$. In order to find the tight Hilbert function, $\ell(R/(I^n)^*)$, we first find the tight closure of $I$ and its powers.
\begin{Theorem}
	Let $\Delta$ be a $(d-1)$-dimensional simplicial complex on $n$ vertices. Let $R=\kk[\Delta]$ be the corresponding Stanley-Reisner ring, $\m$ be its unique maximal homogeneous ideal and $I_1, I_2,\dots, I_g$ be ideals generated by a linear system of parameters in $R$.\\
	{\rm(1)} Then $I_i^* = \m$, for all $i=1,\dots,g$.\\
	{\rm(2)} Let $s_1,s_2,\dots, s_g\in \mathbb{N}$. Then $(I_1^{s_1}I_2^{s_2}\cdots I_g^{s_g})^* = \m^{s_1+s_2+\cdots+s_g}.$
\end{Theorem}
\begin{proof}
	(1) Fix $i \in \{1,\ldots,g\}.$ For an element $x \in R$, we know that $x \in I^*$ if and only if the residue class of $x$ lies in $((I+p)/p)^*,$ for all minimal prime ideals $p$ of $R.$ Thus,
	\[ I_i^*=\{ x \in R \mid \ov{x} \in ((I_i+P_F)/P_F)^* = \m_F^*\kk[F], \text{ for every facet $F$ of } \Delta \}. \]
	Thus $x_1,\dots,x_n \in I_i^*$. This implies that $\m \subseteq I_i^* \subseteq \m$. Hence, $I_i^* = \m$.
	
	(2)	Consider
	\[ (I_1^{s_1}I_2^{s_2}\cdots I_g^{s_g})^* = ((I_1^*)^{s_1}(I_2^*)^{s_2}\cdots (I_g^*)^{s_g})^* = (\m^{s_1+s_2+\cdots+s_g})^* = \m^{s_1+s_2+\cdots+s_g}. \]
\end{proof}

We are now ready to calculate the tight Hilbert polynomial of power product of ideals generated by a  linear system of parameters. Let $h(\Delta)=(h_0, h_1,\dots ,h_d)$ denote the $h$-vector of $\kk[\Delta]$. The Hilbert series of $R$ is of the form
\[ H(\kk[\Delta],\lambda) = \frac{h_0+h_1\lambda+\cdots +h_d\lambda^d}{(1-\lambda)^d}. \]
Set $\h(\lambda)=h_0+h_1\lambda+\cdots +h_d\lambda^d$. Let $\h^{(i)}(\lambda)$ denote the $i^{th}$ derivative of $\h(\lambda)$ with respect to $\lambda$.

\begin{Theorem}\label{faceringpoly}
	Let $\Delta$ be a $(d-1)$-dimensional simplicial complex on $n$ vertices and $\kk$ be field of positive prime characteristic $p$. Let $R=\kk[\Delta]$ be the corresponding Stanley-Reisner ring, $\m$ be its unique maximal homogeneous ideal and $I_1, I_2,\dots, I_g$ be the ideals generated by a  linear system of parameters in $R$. Let $h(\Delta)=(h_0, h_1,\dots ,h_d)$ denote the $h$-vector of $R$. Let $s_1,s_2,\dots, s_g\in \mathbb{N}$ with $s_1+s_2+\cdots+s_g = n+1$. Then 
	\[ \ell\left(\frac{R}{(I_1^{s_1}I_2^{s_2}\cdots I_g^{s_g})^*}\right) = \ell\left(\frac{R}{\m^{n+1}}\right) = \sum_{i=0}^{d}(-1)^i\frac{\h^{(i)}(1)}{i!} {n+d-i\choose d-i}. \]
	In particular, $e_d^*(I_1^{s_1}I_2^{s_2}\dots I_g^{s_g}) = h_d$, for any $s_1,s_2,\dots,s_g\in \mathbb{N}$.
\end{Theorem}

\begin{proof}
	As $R\simeq \gr_\m(R) = \oplus_{n\ge 0}\m^n/\m^{n+1}$ and as Hilbert function and Hilbert polynomial of $R$ coincide for all $n \ge 1$. Using {\cite[Proposition 4.1.9]{brunsHerzog}} for all $n\ge 1$, we have
	
	\begin{equation*}
	l\left(\frac{\m^n}{\m^{n+1}}\right)=\sum_{i=0}^{d-1}(-1)^i\frac{\h^{(i)}(1)}{i!} {n+d-1-i\choose d-1-i}.
	\end{equation*}
	This implies that for all $n\ge 1$,
	
	\begin{align*}
	\ell\left(\frac{R}{\m^{n+1}}\right)
	&=\ell\left(\frac{R}{\m}\right)+\sum_{j=1}^{n}\ell\left(\frac{\m^j}{\m^{j+1}}\right)\\
	&=1+\sum_{j=1}^{n}\sum_{i=0}^{d-1}(-1)^i\frac{\h^{(i)}(1)}{i!} {j+d-1-i\choose d-1-i}\\
	&=1+\sum_{i=0}^{d-1}(-1)^i\frac{\h^{(i)}(1)}{i!} \sum_{j=1}^{n} {j+d-1-i\choose j}\\
	&=1+\sum_{i=0}^{d-1}(-1)^i\frac{\h^{(i)}(1)}{i!} \left[{n+(d-1-i)+1\choose n}-1\right]\\
	&=1+\sum_{i=0}^{d-1}(-1)^i\frac{\h^{(i)}(1)}{i!} \left[{n+d-i\choose d-i}-1\right].
	\end{align*}
	Re-arranging the terms, we get
	\begin{align*}
	\ell\left(\frac{R}{\m^{n+1}}\right)
	&=1-\left(\sum_{i=0}^{d-1}(-1)^i\frac{\h^{(i)}(1)}{i!}\right)+\sum_{i=0}^{d-1}(-1)^i\frac{\h^{(i)}(1)}{i!} {n+d-i\choose d-i}.\\
	\end{align*}
	Thus for $0\le i\le d-1$, $e_i=\h^{(i)}(1)/i!$ and 
	\[(-1)^de_d = 1-\left(\sum_{i=0}^{d-1}(-1)^i\frac{\h^{(i)}(1)}{i!}\right). \]
	Observe that 
	\[ h_0+h_1\lambda+\cdots +h_d\lambda^d = \sum_{i=0}^{d}\frac{(-1)^i(1-\lambda)^i\h^{(i)}(1)}{i!}. \]
	Substituting $\lambda=0$, we get 
	\[ h_0-(-1)^dh_d= \sum_{i=0}^{d-1}\frac{(-1)^i\h^{(i)}(1)}{i!}= 1-(-1)^de_d. \]
	As $h_0=1$, it follows that $e_d(\m)=h_d$ giving us the required result.
\end{proof}
Recall the following result of Rees \cite[Theorem 2.5]{rees1981hilbert}.
\begin{Theorem}[Rees] 
	Let $(R,\m)$ be a 2-dimensional Cohen-Macaulay local ring, $\kk=R/\m$ be infinite. Let $I,J$ be $\m$-primary ideals. Then for any good joint reduction $(a,b)$ of the filtration $\I=\{ \ov{I^rJ^s} \}$, the following are equivalent:\\
	{\rm(1)} $\overline{e_2}(IJ)=\overline{e_2}(I)+\overline{e_2}(J)$.\\
	{\rm(2)} $a\overline{I^{r-1}J^s}+b\overline{I^rJ^{s-1}}=\overline{I^rJ^s}$, for all $r,s\ge 1$.
\end{Theorem}

\noindent
An analogue of this theorem in dimension $3$ was proved in \cite{mpv}. However, no such  analogue is known in dimension $d\geq 4.$   We generalize this theorem  for ideals generated by linear systems of parameters in Stanley-Reisner rings of  simplicial complexes.

\begin{Theorem}
	Let $\kk[\Delta]$ be a $d$-dimensional Cohen-Macaulay Stanley-Reisner ring of a  simplicial complex $\Delta$. Let $\m$ be its unique maximal homogeneous ideal and $I_1,I_2,\dots ,I_d$ be ideals generated by linear systems of parameters. Then for any joint reduction $(a_1,a_2,\dots ,a_d)$ of the $d$-tuple $(I_1,I_2,\dots,I_d)$, the following are equivalent:\smallskip\\
	{\rm(1)} $\Delta$ is Eulerian.\\
	{\rm(2)} $r(\m)\le d-1.$\\
	{\rm(3)} $a_1 (\underline{I}^{S_1})^* + a_2(\underline{I}^{S_2})^* + \cdots + a_d(\underline{I}^{S_d})^* = (\underline{I}^S)^*$, for all $s_1,s_2,\dots ,s_d\ge 1$.\\
	{\rm(4)} $\displaystyle e^*_d(I_1I_2\cdots I_d)-\sum_{i=1}^{d}\left(e^*_d \big( \prod_{j\not=i}I_j \big)\right)+\sum_{1\le p<q\le d}\left(e^*_d \big(\prod_{j\not=p,q}I_j \big)\right)-\cdots+(-1)^{d-1} \sum_{i=1}^{d} e^*_d(I_i)=0.$
\end{Theorem}

\begin{proof}
	$(1\Leftrightarrow 2)$ We know that $\Delta$ is Eulerian if and only if the Hilbert function and the Hilbert polynomial of $\kk[\Delta]$ agree for all $n\ge 0$. This is true if and only if the postulation number of $\kk[\Delta]$ is $\le -1$. Using Marley's result \cite[Corollary 3.8]{MarleyThesis}, as $\kk[\Delta]$ is Cohen-Macaulay, it is true if and only if reduction number of $\kk[\Delta]$ is $\le d-1$. \\
	$(2\Leftrightarrow 3)$ Under the given conditions, 
	\begin{align*}
	& \ a_1(\underline{I}^{S_1})^*+a_2(\underline{I}^{S_2})^*+\cdots +a_d(\underline{I}^{S_d})^*=(\underline{I}^S)^*, \text{ for all } s_1,s_2,\dots ,s_d\ge 1\\ 
	\Leftrightarrow 
	& \ a_1\m^{|S_1|}+a_2\m^{|S_2|}+\cdots +a_d\m^{|S_d|}=\m^{|S|}, \text{ for all } s_1,s_2,\dots ,s_d\ge 1\\
	\Leftrightarrow 
	& \ (a_1,a_2,\dots ,a_d)\m^{|S|-1}=\m^{|S|}, \text{ for all } s_1,s_2,\dots ,s_d\ge 1\\
	\Leftrightarrow 
	& \ r(\m)\le d-1.
	\end{align*}
	$(1\Leftrightarrow 4)$ As $e^*_d(I^rJ^s)=h_d$, for any ideals $I,J$ generated by linear systems of parameters and for any $r,s\in \mathbb{N}$, we get
	\begin{align*}
	& \ e^*_d(I_1\cdots I_d)
	-\sum_{i=1}^{d}\left[e^*_d(\prod_{j\not=i}I_j)\right]+\sum_{1\le p<q\le d}\left[e^*_d(\prod_{j\not=p,q}I_j)\right]-\cdots +(-1)^{d-1} \sum_{i=1}^{d}e^*_d(I_i)=0\\
	\Leftrightarrow & \ h_d-dh_d+{d \choose 2}h_d-\cdots +(-1)^{d-1}dh_d=0\\
	\Leftrightarrow & \ \left[1-{d \choose 1}+{d \choose 2}-\cdots + (-1)^{d-1}{d \choose 1} + (-1)^d+(-1)^{d-1}\right]h_d=0\\
	\Leftrightarrow & \ (-1)^{d-1}h_d=0\\
	\Leftrightarrow & \ h_d=0\\
	\Leftrightarrow & \ \chi(\Delta)=1.
	\end{align*}
\end{proof}

\begin{Remark}{\rm
	We observe here that if $I_1, I_2,\dots, I_g$ are the ideals generated by linear systems of parameters in $R=\kk[\Delta]$ and  $s_1,s_2,\dots, s_g\in \mathbb{N}$, then $\m^{s_1+s_2+\cdots+s_g} = \ov{I_1^{s_1}I_2^{s_2}\cdots I_g^{s_g}}$ irrespective of the characteristic. For all $i=1,\dots,g$, $I_i$ is a reduction of $\m$ implies that $\ov{I^{s_i}_i} = \ov{\m^{s_i}} = \m^{s_i}$. This implies that 
	\[ \ov{I_1^{s_1}I_2^{s_2}\cdots I_g^{s_g}} = \ov{\ov{I_1^{s_1}}\ \ov{I_2^{s_2}}\cdots \ov{I_g^{s_g}}} = \ov{\m^{s_1+s_2+\cdots+s_g}} = \m^{s_1+s_2+\cdots+s_g}. \]
	Thus in case of integral closure, calculations in the Theorem \ref{faceringpoly} are independent of the characteristic. In particular, $\ov{e}_d (I_1^{s_1}I_2^{s_2}\cdots I_g^{s_g}) = h_d$, for any $s_1,s_2,\dots,s_g\in \mathbb{N}$. Hence, the above proposition is true when tight closure is replaced by integral closure.
}\end{Remark}

\bibliographystyle{plain}
\bibliography{gmvrevised}     

\end{document}